\documentclass[reqno,12pt]{amsart}
\usepackage{amssymb,amsmath}
\usepackage{tikz,enumerate,hyperref}
\usepackage{ifthen}

\usetikzlibrary{patterns}
\usetikzlibrary{decorations.pathreplacing}
\definecolor{tileColor}{HTML}{FF3300}

\newtheorem{theorem}{Theorem}
\newtheorem{lemma}[theorem]{Lemma}

\theoremstyle{definition}
\newtheorem{definition}[theorem]{Definition}

\newcommand{\oeis}[1]{\href{https://oeis.org/#1}{#1}}

\title{Generalized metallic means}
\author{Juan B. Gil}
\author{Aaron Worley}
\address{Penn State Altoona\\ 3000 Ivyside Park\\ Altoona, PA 16601}

\begin{document}
\maketitle

\begin{abstract}
The metallic means (also known as metallic ratios) may be defined as the limiting ratio of consecutive terms of sequences connected to the Fibonacci sequence via the {\sc invert} transform. For example, the Pell sequence ({\sc invert} transform of the Fibonacci sequence) gives the so-called silver mean, and the {\sc invert} transform of the Pell sequence leads to the bronze mean. The limiting ratio of the Fibonacci sequence itself is known as the golden mean or ratio. We introduce a new family of $k$th-degree metallic means obtained through {\sc invert} transforms of the generalized $k$th-order Fibonacci sequence. As it is the case for $k=2$, each generalized metallic mean is shown to be the unique positive root of a $k$th-degree polynomial determined by the sequence. 
\end{abstract}

\thispagestyle{empty}
\section{Introduction}

We start by reviewing some basic information about the sequence of Fibonacci numbers, defined through the recurrence relation
\begin{gather*}
 F_0=0,\; F_1=1,\\
 F_n = F_{n-1}+F_{n-2} \;\text{ for } n\ge 2.
\end{gather*}
Their generating function $F(x)=\sum_{n=1}^{\infty} F_nx^n$ may be written as $F(x)=\dfrac{x}{1-x-x^2}$, and it can be easily shown that
\[ \lim_{n\to\infty}{\frac{F_{n+1}}{F_{n}}} = \varphi, \]
where $\varphi$ is the unique positive root of $x^2 -x -1 =0$. The number $\varphi = \frac{1+\sqrt{5}}{2} \approx 1.618$ is referred to as the Golden Mean or Golden Ratio.

More generally, for $m\in\mathbb{N}$, one can consider the sequence $\big(F^{(m)}_n\big)_{n\in\mathbb{N}_0}$ defined by 
\begin{gather*}
 F^{(m)}_0=0,\; F^{(m)}_1=1,\\
 F^{(m)}_n = mF^{(m)}_{n-1}+F^{(m)}_{n-2} \;\text{ for } n\ge 2,
\end{gather*}
with generating function 
\[ F_m(x) = \frac{x}{1-mx-x^2}. \] 
Observe that $\big(F^{(2)}_n\big)$ is the sequence of Pell numbers (see \cite[\oeis{A000129}]{Sloane}), and for $m=3$ and $m=4$, we get sequences \oeis{A006190} and \oeis{A001076} in \cite{Sloane}.

It can be shown that 
\[ \lim_{n\to\infty} \frac{F^{(m)}_{n+1}}{F^{(m)}_n} = \varphi_m, \]
where $\varphi_m$ is the unique positive root of $x^2 -mx -1 =0$. Thus
\[ \varphi_m = \frac{m + \sqrt{m^2 +4}}{2}. \]
Note that $\varphi_m$ lies in the interval $(m,m+1)$.
These numbers are collectively referred to as the (quadratic) {\em metallic means}. For instance, $\varphi_2 = 1 + \sqrt{2} \approx 2.414$ is known as the Silver Mean, and $\varphi_3 = \frac{3 + \sqrt{13}}{2} \approx 3.303$ is called the Bronze Mean.  

The above family of sequences may also be generated via the {\sc invert} transform\footnote{Introduced by P.~J.~Cameron in \cite[Section~3]{Cameron} as operator $A$.}. In fact, for every $m\in\mathbb{N}$, we have
\begin{equation*}
 1 + F_{m+1}(x) = \frac{1}{1-F_m(x)}.
\end{equation*}
In other words, the sequence $\big(F^{(m+1)}_n\big)$ is the {\sc invert} transform of the sequence $\big(F^{(m)}_n\big)$, hence it is the $m$th {\sc invert} transform of the Fibonacci sequence.

In this paper, we adopt this point of view to introduce a family of generalized metallic means obtained through {\sc invert} transforms of the generalized $k$th-order Fibonacci sequence. In Section~\ref{sec:cubic}, we illustrate our approach for the cubic case ($k=3$) that involves the tribonacci sequence. In Section~\ref{sec:general}, we then generalize our results to arbitrary values of $k$. In all cases, each generalized metallic mean is shown to be the unique positive root of a polynomial determined by the sequence. In the last section, we offer some remarks regarding combinatorial interpretations and a possible direction for future research.

\section{Cubic metallic means}
\label{sec:cubic}

We start by considering the sequence $(T_n)_{n\in\mathbb{N}_0}$ defined by the recurrence
\begin{gather*}
T_0 = 0,\; T_1=1,\; T_2=1, \\
T_n = T_{n-1} + T_{n-2} + T_{n-3} \;\text{ for } n\ge 3.
\end{gather*}
The corresponding generating function $T(x)$ takes the form
\[ T(x) = \frac{x}{1-x-x^2-x^3}, \]
and its $(m-1)$-th {\sc invert} transform is the sequence $\big(T^{(m)}_{n}\big)_{n\in\mathbb{N}_0}$ with generating function
\[ T_m(x) = \frac{x}{1-mx-x^2-x^3}. \]
In particular, $\big(T^{(m)}_n\big)_{n\in\mathbb{N}_0}$ satisfies the recurrence relation
\begin{gather*}
T^{(m)}_{0} = 0,\; T^{(m)}_{1}=1,\; T^{(m)}_{2}=m, \\
T^{(m)}_{n} = mT^{(m)}_{n-1} + T^{(m)}_{n-2} + T^{(m)}_{n-3} \;\text{ for } n\ge 3.
\end{gather*}

Observe that the corresponding characteristic polynomial $p(x)=x^3 - mx^2 -x -1$ satisfies $p(m)<0<p(m+1)$, so $p(x)$ has a real root $\tau_m\in(m,m+1)$. Moreover,
\[ p(x) = (x-\tau_m)\left(x^2+\tfrac{\tau_m+1}{\tau_m^2}\,x+\tfrac{1}{\tau_m}\right), \]
hence the other two roots of $p(x)$ are the complex numbers 
\[ \gamma_m^{\pm} = -\frac{1}{2\tau_m^2}\left(\tau_m+1 \pm i\sqrt{4\tau_m^3-(\tau_m+1)^2}\right). \]
Since $|\gamma_m^{\pm}|=1/\sqrt{\tau_m}<\tau_m$, and since 
\[  T^{(m)}_{n} = c_1 \tau_m^n + c_2 (\gamma_m^-)^n + c_3 (\gamma_m^+)^n 
    \;\text{ for some constants } c_1,c_2,c_3, \]
one can easily deduce that 
\[ \lim_{n\to\infty} \frac{T^{(m)}_{n+1}}{T^{(m)}_n} = \tau_m. \]
This motivates the following definition.

\begin{definition}
Let $m\in\mathbb{N}$. The unique real root $\tau_m$ of the polynomial $x^3 - mx^2 -x -1$ will be referred to as the $m$th {\em cubic metallic mean}. We have $m<\tau_m<m+1$.
\end{definition}
For example, for $m=1,2,3$, we get
\begin{align*} 
\tau_1 &= \frac{1}{3}\left(1 + \sqrt[3]{19+3\sqrt{33}} + \sqrt[3]{19-3\sqrt{33}}\right)
  \approx 1.839, \\
\tau_2 &= \frac{1}{3}\left(2 + \sqrt[3]{\tfrac{61+9\sqrt{29}}{2}} + \sqrt[3]{\tfrac{61-9\sqrt{29}}{2}}\right)
  \approx 2.547, \\
\tau_3 &= \frac{1}{3}\left(3 + \sqrt[3]{54 + 6\sqrt{33}} + \sqrt[3]{54 - 6\sqrt{33}} \right) 
  \approx 3.383.
\end{align*}

\medskip\noindent
We call $\tau_1$,  $\tau_2$, and $\tau_3$ the {\em cubic} golden, silver, and bronze means, respectively. The number $\tau_1$ is also known as the {\em tribonacci constant}, see e.g.\ \cite{Sharp}. For more information, including geometric interpretations of $\tau_1$, we refer to \cite[\oeis{A058265}]{Sloane} and the links therein. For example, one interpretation of $\tau_1$ is the following. If a line segment is divided into three parts of lengths $a$, $b$, and $c$ such that $\frac{a+b+c}{a} = \frac{a}{b} = \frac{b}{c}$, then the common ratio is precisely $\tau_1$ (see Fig.~\ref{fig:golden}).

\bigskip
\begin{figure}[ht]
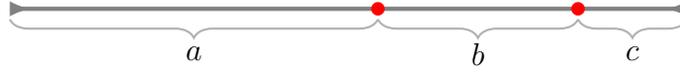

\tikz[scale=0.9]{%
\draw[ultra thick,gray,latex reversed-latex reversed] (0,0)--(10,0); 
\fill[red] (5.437,0) circle (0.1); \fill[red] (8.393,0) circle (0.1);
\draw [decorate,decoration={brace,amplitude=7pt,mirror,raise=4pt},thick,gray!60]
 (0,0) -- (5.437,0) node [midway,yshift=-6mm,black] {$a$};
\draw [decorate,decoration={brace,amplitude=7pt,mirror,raise=4pt},thick,gray!60]
 (5.437,0) -- (8.393,0) node [midway,yshift=-6mm,black] {$b$};
\draw [decorate,decoration={brace,amplitude=7pt,mirror,raise=4pt},thick,gray!60]
 (8.393,0) -- (10,0) node [midway,yshift=-6mm,black] {$c$};
}
\caption{$(a+b+c):a$ is equal to $a:b$ and equal to $b:c$.}
\label{fig:golden}
\end{figure}

It is worth noting that $\tau_1$ also appears as the order of convergence of a certain iterative algorithm for solving nonlinear least squares problems, cf.~\cite{ShGna05}.

\section{Generalized metallic means} 
\label{sec:general}

Motivated by the quadratic and cubic metallic means defined via the {\sc invert} transform of the Fibonacci and tribonacci sequences, respectively, we now consider the generalized $k$th-order Fibonacci sequence ($k\ge 2$) with generating function
\[ G(x) = \frac{x}{1-x-x^2- \cdots - x^k}. \]
If we write $G(x)=\sum_{n=1}^\infty g_n x^n$, then $(g_n)_{n\in\mathbb{N}_0}$ satisfies the recurrence relation
\begin{equation*}
g_n = g_{n-1} + g_{n-2} + \dots + g_{n-k} \;\text{ for } n\ge k,
\end{equation*}
with initial values $g_0=0$, $g_1=1$, and if $k\ge 3$, 
\[ g_j = 2^{j-2} \;\text{ for } j\in\{2,\dots,k-1\}. \]
We apply the {\sc invert} transform $m-1$ times and arrive at the sequence $\big(g^{(m)}_n\big)_{n\in\mathbb{N}_0}$ with generating function
\begin{equation} \label{eq:G_m}
 G_m(x) = \frac{x}{1-mx-x^2- \cdots - x^k},
\end{equation}
thus we have the recurrence relation
\[ g^{(m)}_{n} = mg^{(m)}_{n-1} + g^{(m)}_{n-2} + \dots +g^{(m)}_{n-k} \;\text{ for } n\ge k. \]

As a consequence, if $\gamma_1,\dots,\gamma_k$ are the distinct\footnote{This fact will be discussed in the proof of Theorem~\ref{thm:metallicMean}.} roots of the polynomial
\begin{equation} \label{eq:charpoly}
 p_m(x) = x^k-mx^{k-1}-x^{k-2} - \dots - x-1,
\end{equation}
then we have
\begin{equation*}
  g^{(m)}_{n} = c_1\gamma_1^n + \dots + c_k\gamma_k^n,
\end{equation*}
for some constants $c_1,\dots,c_k$.

\begin{theorem}
For every $m\in\mathbb{N}$, the polynomial $p_m(x)$ has a unique positive root $\varrho_m$ with $m<\varrho_m<m+1$. Moreover, every other root $\gamma_j$ of $p_m(x)$ satisfies $|\gamma_j|<\varrho_m$.
\end{theorem}
\begin{proof}
Since $(x-1) p_m(x) = x^{k+1} - (m+1)x^k + (m-1)x^{k-1} +1$, we have
\begin{equation}\label{eq:p_m}
 p_m(x) = \frac{x^{k+1} - (m+1)x^k + (m-1)x^{k-1} +1}{x-1} \; \text{ for } x\not=1.
\end{equation}
Hence
\begin{equation*}
 p_m(m) =	\begin{cases} 
 		1-k &\text{if } m=1, \\
		\frac{1-m^{k-1}}{m-1} & \text{if } m>1,
		\end{cases}
 \quad\text{and}\quad p_m(m+1) = \frac{(m-1)(m+1)^{k-1} +1}{m}.
\end{equation*}
Clearly, $p_m(m)<0<p_{m}(m+1)$ for every $m$, which implies that $p_m(x)$ must have a real root between $m$ and $m+1$. We call this root $\varrho_m$. By Descartes' rule of signs, $\varrho_m$ is the only positive real root of $p_m(x)$. 

As a consequence, $p_m(x)<0$ for every $x\in (0,\varrho_m)$ and $p_m(x)>0$ for every $x>\varrho_m$.
Let $\gamma\not=\varrho_m$ be such that $p_m(\gamma)=0$. Then $\gamma^{k} = m\gamma^{k-1} + \gamma^{k-2} +\cdots+\gamma+1$ and so
\[ |\gamma|^{k} \le m|\gamma|^{k-1} + |\gamma|^{k-2} +\cdots+|\gamma|+1. \]
This implies $p_m(|\gamma|)\le 0$, hence $|\gamma|\le \varrho_m$. We will show that this inequality is strict.

Note that, by \eqref{eq:p_m}, $p(x)=0$ if and only if $(m+1)x^k = x^{k+1} + (m-1)x^{k-1} +1$.

Assume that $\gamma\not=\varrho_m$ is a root with $|\gamma|=\varrho_m$. Thus $p_m(\gamma)=0=p_m(|\gamma|)$, hence
\begin{align} \label{eq:gamma}
 (m+1)\gamma^k &= \gamma^{k+1} + (m-1)\gamma^{k-1} +1, \\ \label{eq:|gamma|}
 (m+1)|\gamma|^k &= |\gamma|^{k+1} + (m-1)|\gamma|^{k-1} +1.
\end{align}
Equation~\eqref{eq:gamma} together with Lemma~\ref{lem:triangle_Ineq} gives
\[ |(m+1)\gamma^k| + \underbrace{\left(3-\left|\tfrac{\gamma^{k+1}}{|\gamma^{k+1}|}+\tfrac{\gamma^{k-1}}{|\gamma^{k-1}|}+1\right|\right)}_{\ge 0} \le |\gamma^{k+1}|+|(m-1)\gamma^{k-1}|+1, \]
which by \eqref{eq:|gamma|} implies $\left|\tfrac{\gamma^{k+1}}{|\gamma^{k+1}|}+\tfrac{\gamma^{k-1}}{|\gamma^{k-1}|}+1\right| = 3$. If $\gamma=\varrho_m e^{i\theta}$, this can be written as
\[ \left|e^{i\theta(k+1)}+e^{i\theta(k-1)}+1\right| = 3, \]
which is only possible if $\theta=0$ or if $\theta=\pi$ (for odd $k$). However, $\theta=0$ would contradict the fact that $\gamma\not=\varrho_m$, and if $\theta=\pi$ and $k$ is odd, the left-hand side of \eqref{eq:gamma} would be negative while the right-hand side would be positive, a contradiction.

In conclusion, if $\gamma\not=\varrho_m$ and $p_m(\gamma)=0$, then $|\gamma| < \varrho_m$.
\end{proof}

The following lemma was motivated by a theorem of Maligranda \cite{Maligranda}. 
\begin{lemma}\label{lem:triangle_Ineq}
Let $z_1,z_2,z_3\in\mathbb{C}$ be such that $0<|z_1|\le |z_2|\le |z_3|$. Then
\[ |z_1+z_2+z_3| + \left(3-\left|\frac{z_1}{|z_1|}+\frac{z_2}{|z_2|}+\frac{z_3}{|z_3|}\right|\right)|z_1| 
  \le |z_1|+|z_2|+|z_3|. \]
\end{lemma}
\begin{proof}
By the triangle inequality, and since $|z_1|\le |z_2|\le |z_3|$, we have
\begin{align*}
 |z_1+z_2+z_3| &= \left|\frac{|z_1|}{|z_1|}z_1+\frac{|z_1|}{|z_2|}z_2+\frac{|z_1|}{|z_3|}z_3 
 +\left(1-\frac{|z_1|}{|z_2|}\right)z_2 + \left(1-\frac{|z_1|}{|z_3|}\right)z_3\right| \\
 &\le \left|\frac{z_1}{|z_1|}+\frac{z_2}{|z_2|}+\frac{z_3}{|z_3|}\right||z_1| 
 + \left(1-\frac{|z_1|}{|z_2|}\right)|z_2| + \left(1-\frac{|z_1|}{|z_3|}\right)|z_3| \\
 &= \left|\frac{z_1}{|z_1|}+\frac{z_2}{|z_2|}+\frac{z_3}{|z_3|}\right||z_1| 
 + |z_2| - |z_1| + |z_3| - |z_1|.
\end{align*} 
Adding and subtracting $|z_1|$, we arrive at
\[ |z_1+z_2+z_3| \le 
  |z_1| + |z_2| + |z_3| + \left|\frac{z_1}{|z_1|}+\frac{z_2}{|z_2|}+\frac{z_3}{|z_3|}\right||z_1| - 3|z_1|, \]
which is equivalent to the claimed inequality.
\end{proof}

\begin{theorem} \label{thm:metallicMean}
Fix $k, m\in\mathbb{N}$, $k\ge 2$, and let $\big(g^{(m)}_n\big)_{n\in\mathbb{N}}$ be the sequence defined by the generating function $G_m(x)$ from \eqref{eq:G_m}. Then
\[ \lim_{n\to\infty} \frac{g^{(m)}_{n+1}}{g^{(m)}_n} = \varrho_m, \]
where $\varrho_{m}$ is the unique positive root of the polynomial \eqref{eq:charpoly}. Consistent with the quadratic and cubic cases, we call $\varrho_{m}$ the $m${\em th} metallic mean of degree $k$.
\end{theorem}
\begin{proof}
We already showed that $p_m(x)$ has a unique positive real root $\varrho_m\in(m,m+1)$, and we claim that every root of $p_m(x)$ is simple. To see this, consider the polynomial
\[ q(x) = (x-1)p_m(x) = x^{k+1} - (m+1)x^k + (m-1)x^{k-1} +1. \] 
Clearly, $p_m(x)$ and $q(x)$ share all of their roots (including multiplicity) except for $x=1$. Now, since 
\[ q'(x) = x^{k-2}\left((k+1)x^2 - k(m+1)x + (k-1)(m-1)\right), \]
and since the quadratic polynomial $(k+1)x^2 - k(m+1)x + (k-1)(m-1)$ has two distinct positive real roots (the discriminant is $k^2(m-1)^2+4(k^2+m-1)$), we conclude that the roots of $q(x)$, and therefore the roots of $p_m(x)$, are all simple. 

Let $\gamma_1,\dots,\gamma_{k-1}, \varrho_m$ be the $k$ distinct roots of the polynomial $p_m(x)$ associated with the sequence $\big(g^{(m)}_n\big)_{n\in\mathbb{N}}$. Then, there are constants $c_1,\dots,c_k$ such that 
\begin{equation*}
  g^{(m)}_{n} = c_1\gamma_1^n + \dots + c_{k-1}\gamma_{k-1}^n + c_k \varrho_m^n.
\end{equation*}
Therefore,
\begin{align*}
  \lim_{n\to\infty} \frac{g^{(m)}_{n+1}}{g^{(m)}_{n}} 
 &=  \lim_{n\to\infty} \frac{c_1\gamma_1^{n+1} + \dots + c_{k-1}\gamma_{k-1}^{n+1} + c_k \varrho_m^{n+1}}{c_1\gamma_1^n + \dots + c_{k-1}\gamma_{k-1}^n + c_k \varrho_m^n} \\
 &=  \lim_{n\to\infty} \frac{c_1\gamma_1(\frac{\gamma_1}{\varrho_m})^{n} + \dots + c_{k-1}\gamma_{k-1}(\frac{\gamma_{k-1}}{\varrho_m})^{n} + c_k \varrho_m}{c_1(\frac{\gamma_1}{\varrho_m})^{n} + \dots + c_{k-1}(\frac{\gamma_{k-1}}{\varrho_m})^{n} + c_k} \\
 &= \varrho_{m}.
\end{align*}
The last step follows from the fact that, for $j\in\{1,\dots,k-1\}$, we have $|\gamma _j| < \varrho_{m}$ which implies $\lim\limits_{n\to\infty} \big(\frac{\gamma _j}{\varrho_{m}}\big)^n = 0$.
\end{proof}

\section{Concluding remarks} 

Motivated by the (quadratic) metallic means, which may be defined through sequences that are related to the Fibonacci sequence via the {\sc invert} transform, in this paper we have introduced a family of generalized  metallic means of arbitrary degree $k>2$. 

Observe that our definition is consistent with the quadratic case ($k=2$), and the generalized metallic means of degree $k$ as well as the corresponding sequences $(g^{(m)}_n)_{n\in\mathbb{N}_0}$, all satisfy similar properties as their quadratic counterparts. 

Combinatorially, it is known and easy to prove that, for $n\ge 2$, the sequence $F_n$ gives the number of tilings of an $(n-1)\times 1$ rectangular board by $1\times 1$ and $2\times 1$ tiles. In that context, $F^{(m)}_n$ gives the number of such tilings, where the $1\times 1$ tiles come in $m$ colors. In general, for $k,n\ge 2$, $g^{(m)}_n$ gives the number of tilings of an $(n-1)\times 1$ rectangular board by tiles of sizes $1\times 1$, $2\times 1$, \dots, $k\times 1$, where the $1\times 1$ tiles come in $m$ colors.

For example, if $k=3$ and $m=2$, we get the sequence $(T^{(2)}_n)_{n\in\mathbb{N}}$ with terms
\[ 1, 2, 5, 13, 33, 84, 214, 545, 1388, 3535, 9003, 22929, 58396, \dots \]
The 13 such tilings of a $3\times 1$ rectangular board are:

\bigskip
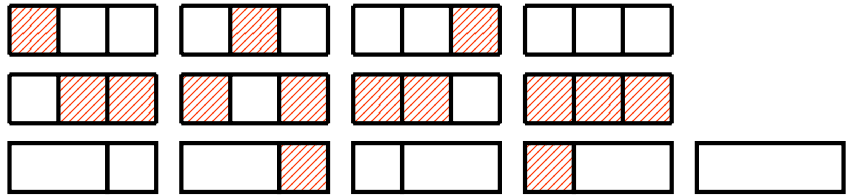
\begin{center}
\tikzstyle{cTile}=[pattern=north east lines, pattern color=tileColor]
\begin{tikzpicture}[scale=0.65]
\foreach \x in {0,1}{
\begin{scope}[xshift=100*\x]
\ifthenelse{\x=1}{\draw[cTile] (2,0) rectangle (3,1);}{}
\draw[ultra thick] (0,0) rectangle (2,1); 
\draw[ultra thick] (2,0) rectangle (3,1); 
\end{scope}
}
\foreach \x in {2,3}{
\begin{scope}[xshift=100*\x]
\ifthenelse{\x=3}{\draw[cTile] (0,0) rectangle (1,1);}{}
\draw[ultra thick] (0,0) rectangle (1,1); 
\draw[ultra thick] (1,0) rectangle (3,1); 
\end{scope}
}
\begin{scope}[xshift=400]
\draw[ultra thick] (0,0) rectangle (3,1); 
\end{scope}
\foreach \x in {0,1,2,3}{
\begin{scope}[xshift=100*\x, yshift=40]
\draw[cTile] (0,0) rectangle (3,1);
\ifthenelse{\x=3}{}{\draw[fill=white] (\x,0) rectangle (\x+1,1);}
\draw[ultra thick] (0,0) grid (3,1); 
\end{scope}
}
\foreach \x in {0,1,2,3}{
\begin{scope}[xshift=100*\x, yshift=80]
\ifthenelse{\x=3}{}{\draw[cTile] (\x,0) rectangle (\x+1,1);}
\draw[ultra thick] (0,0) grid (3,1); 
\end{scope}
}
\end{tikzpicture}
\end{center}
\bigskip

Other combinatorial interpretations are certainly possible.

\medskip
We finish our exposition by mentioning that, for the generalized golden mean ($m=1$) of arbitrary degree, Hare, Prodinger, and Shallit \cite{HPS} gave series representations for $\varrho_1$, $1/\varrho_1$, and $1/(2-\varrho_1)$. It would be interesting to find corresponding series representations for the $m$th generalized metallic mean $\varrho_m$.


\end{document}